\documentclass[11pt,a4paper]{amsart}
\usepackage[T1]{fontenc}
\usepackage[utf8]{inputenc}
\usepackage{lmodern}
\usepackage[a4paper,textwidth=160mm,textheight=238mm,centering]{geometry}
\usepackage{amsmath,amssymb,amsthm,mathtools}
\usepackage{microtype}
\usepackage{xcolor}
\usepackage[colorlinks=true,linkcolor=blue!45!black,citecolor=blue!45!black,urlcolor=blue!45!black]{hyperref}
\hypersetup{pdftitle={Singularity of harmonic measure on Poisson--Voronoi tessellations in hyperbolic spaces},pdfauthor={Kohki Sakamoto}}

\newtheorem{theorem}{Theorem}[section]
\newtheorem{proposition}[theorem]{Proposition}
\newtheorem{lemma}[theorem]{Lemma}
\newtheorem{corollary}[theorem]{Corollary}
\newtheorem*{conjecture}{Conjecture}
\numberwithin{equation}{section}
\theoremstyle{definition}
\newtheorem{definition}[theorem]{Definition}

\theoremstyle{remark}
\newtheorem{remark}[theorem]{Remark}

\newcommand{\HH}{\mathbb H^d}
\newcommand{\PP}{\mathbb P}
\newcommand{\EE}{\mathbb E}
\newcommand{\RR}{\mathbb R}
\newcommand{\ZZ}{\mathbb Z}
\newcommand{\cE}{\mathcal E}
\newcommand{\cR}{\mathcal R}
\newcommand{\one}{\mathbf 1}
\newcommand{\vol}{\operatorname{vol}}
\newcommand{\Isom}{\operatorname{Isom}}
\newcommand{\dd}{\mathop{}\!\mathrm d}

\title[Singularity on hyperbolic Poisson--Voronoi tessellations]{Singularity of harmonic measure on Poisson--Voronoi tessellations in hyperbolic spaces}
\author{Kohki Sakamoto}
\address{Graduate School of Mathematical Sciences, University of Tokyo, 3-8-1 Komaba, Meguro, Tokyo 153-8914, Japan}
\email{sakamoto-kohki571@g.ecc.u-tokyo.ac.jp}
\date{}

\begin{document}
\begin{abstract}
We prove that the harmonic measure of simple random walk on the Poisson--Voronoi tessellation in $\mathbb H^d$ is almost surely singular with respect to Lebesgue measure on the boundary. This answers a question of Benjamini, Paquette, and Pfeffer in the hyperbolic plane and establishes the analogous result in higher dimensions. Our proof is based on ergodic theory of rerooting relations and their boundary extensions.
\end{abstract}
\maketitle

\section{Introduction}
Poisson–Voronoi tessellations provide natural isometry-invariant random discretizations of homogeneous spaces. Given a Poisson point process on such a space, its Voronoi tessellation partitions the space into cells according to the nearest point of the configuration, and its dual graph is called the Delaunay graph. These tessellations and graphs are central objects of stochastic geometry, in particular on Euclidean spaces. See \cite{Mol94} for background.

In this paper, we focus on simple random walk on Poisson--Delaunay graphs. Let $d\ge2$ and $\lambda>0$, and consider a Poisson point process on $\mathbb H^d$ with intensity $\lambda\vol$, where $\vol$ denotes hyperbolic volume. For a realization $\omega$ of the Poisson point process, write $\mathcal G_\omega$ for its Delaunay graph. Benjamini, Paquette, and Pfeffer \cite{BPP18} proved that simple random walk on $\mathcal G_\omega$ in the hyperbolic plane has positive speed, meaning that the distance from the origin grows linearly in time. Later, Paquette \cite{Paq18} and Benjamini--Krauz--Paquette \cite{BKP21} generalized this to higher dimensions. Combining this positive speed with Kaimanovich's criterion \cite{Kai00}, we know that the walk converges almost surely to a point of $\partial\mathbb H^d$, and hence we can consider the distribution of the limiting point on the boundary, called the \emph{harmonic measure} of the walk, denoted by $\nu_\omega$.

As we will see in Subsection~\ref{subsec:related-work}, a fundamental problem in the study of random walks in negatively curved spaces is to compare harmonic measure with natural geometric measures on the boundary. In our setting, the natural reference measure is the Lebesgue measure on $\partial\mathbb H^d\simeq\mathbb S^{d-1}$, and  Benjamini, Paquette, and Pfeffer \cite{BPP18} asked the following.

\begin{conjecture}[Conjecture~1.10, \cite{BPP18}]
For every $\lambda>0$, the harmonic measure of simple random walk on the Poisson--Delaunay graph in $\mathbb H^2$ with intensity $\lambda$ is almost surely singular with respect to Lebesgue measure on $\partial\mathbb H^2$.
\end{conjecture}

Our main result proves this conjecture and its higher-dimensional analogue.

\begin{theorem}\label{thm:main}
Let $d\ge2$ and $\lambda>0$. For almost every realization $\omega$ of the Poisson point process with intensity measure $\lambda\vol$ in $\mathbb H^d$, the harmonic measure $\nu_\omega$ of simple random walk on the associated Delaunay graph is singular with respect to Lebesgue measure on $\partial\mathbb H^d$.
\end{theorem}

Note that, by rotational invariance, the averaged version of the hitting measures is just Lebesgue measure. Theorem~\ref{thm:main} shows that, for almost every realization, the harmonic measure is nevertheless concentrated on a set of Lebesgue measure zero.

The conjecture was motivated by the dimension-drop phenomenon discovered by Lyons, Pemantle, and Peres \cite{LPP95}. They proved that, for a nondegenerate supercritical Galton--Watson tree with finite mean, conditioned on survival, the harmonic measure of simple random walk is exact dimensional and has strictly smaller Hausdorff dimension than the boundary. Given that the harmonic measure of the simple random walk on a regular tree is equal to the Hausdorff measure on the boundary, their result illustrates how harmonic measure can differ from a natural geometric measure on the boundary, when the environment is randomized.

\subsection{Related work}\label{subsec:related-work}
First, we review further results on the comparison of harmonic and geometric measures in negatively curved spaces. For Brownian motion on universal covers of compact negatively curved manifolds, Ledrappier \cite{Led88} characterized equivalence of harmonic and Patterson--Sullivan measures by equality in the fundamental inequality. For compact negatively curved surfaces, he proved that this equivalence forces the curvature to be constant \cite{Led90}.

For finitely supported symmetric random walks on nonelementary hyperbolic groups, Blach\`ere, Ha\"issinsky, and Mathieu \cite{BHM11} established an analogous characterization: equality in the fundamental inequality is equivalent to equivalence of harmonic and Patterson--Sullivan measures. Gou\"ezel, Math\'eus, and Maucourant \cite{GMM18} proved that the inequality is strict for non-virtually-free hyperbolic groups equipped with word metrics.

The comparison of harmonic and geometric measures also arises for random walks on hyperbolic lattices, where the geometric reference is Lebesgue measure on $\partial\mathbb H^d$. Such random walks can be viewed as walks on periodic tessellations of $\mathbb H^d$, providing deterministic counterparts of the Poisson--Voronoi model considered here. Kaimanovich and Le Prince \cite{KLP11} conjectured that every finitely supported random walk generating a hyperbolic lattice has harmonic measure singular with respect to Lebesgue measure. The nonuniform case was studied by Guivarc'h and Le Jan \cite{GLJ93}; see also \cite{GMT15,RT21} and references therein. Recently, the conjecture was settled for cocompact Fuchsian groups, independently by B\'enard \cite{Ben26} and Kosenko and Tiozzo \cite{KT26}. Bogachev \cite{Bog26} subsequently proved the conjecture for cocompact Kleinian groups. It remains open in general in dimensions $d\ge4$. See also \cite{LTVL25} and \cite{KZ25} for results on discrete subgroups in higher rank.

Benjamini, Krauz, and Paquette \cite{BKP21} proved anchored expansion of Poisson--Delaunay graphs in every dimension, which implies positive hyperbolic speed for simple random walk on these graphs. Carrasco, Lessa, and Paquette \cite{CLP17} studied the speed and asymptotic entropy of the walk as $\lambda\to0$. As a consequence, they obtained dimension drop for sufficiently small $\lambda$ in every dimension, without giving an explicit threshold. The author \cite{Sak24} subsequently established exact dimensionality and the dimension formula for harmonic measures, as part of a more general result for invariant random graphs in Gromov hyperbolic spaces. Theorem~\ref{thm:main} establishes singularity at every intensity without using quantitative estimates on the speed or asymptotic entropy.

Hyperbolic Poisson--Voronoi tessellations have also been studied from more geometric perspectives. Benjamini and Schramm \cite{BS01} investigated Bernoulli percolation on Poisson--Voronoi tessellations in the hyperbolic plane and found behavior different from the Euclidean case. Motivated by their work, Hansen and M\"uller \cite{HM22,HM24} determined the high and low intensity asymptotics of the critical probability for Bernoulli percolation. Recently, D'Achille, Curien, Enriquez, Lyons, and \"Unel \cite{DCELU26} constructed the ideal Poisson--Voronoi tessellation for hyperbolic spaces, which arises as a nontrivial limit of Poisson--Voronoi tessellations when the intensity goes to zero. Similar objects are also constructed in higher rank symmetric spaces \cite{FMW23}, with applications to measured group theory.

\subsection{Outline of the proof}

Our proof builds on the ergodic theory of rerooting relations for invariant point processes developed by Ab\'ert and Mellick \cite{AM22}; see also \cite{FMW23}. The boundary extension of the rerooting relation also played a key role in our proof of the dimension formula for harmonic measure \cite{Sak24}. Here we consider its extension to pairs of boundary points and prove its ergodicity using the shift on the bundle of bilateral random walks. This double ergodicity yields a dichotomy: harmonic measure is almost surely either absolutely continuous with respect to Lebesgue measure or singular with respect to it; we rule out the former.

Another crucial input is the Na\"im current, a Radon measure on the double boundary. Its measure class agrees with that of the product of harmonic measure with itself, while the Doob--Na\"im formula relates it to Dirichlet energies of harmonic extensions \cite{Sil74}. Under absolute continuity, double ergodicity imposes rigidity on the Na\"im current through comparison with the Liouville current.

We show that this rigidity is incompatible with insertion tolerance of the Poisson process by comparing Dirichlet energies under suitable finite insertions. In dimension two, harmonic extensions of smooth boundary functions have finite energy, so the Doob--Na\"im formula applies directly. In higher dimensions, these energies may be infinite. Instead, we use trace forms on finite sets and their relation to the Na\"im current to isolate the energy change caused by a finite insertion.

\subsection{Organization of the paper}
In Section 2, we review some basic properties of Poisson point processes and random walks, and prepare the ergodic-theoretic framework. In Section 3, we construct the Naïm current and prove its rigidity under absolute continuity. In Section 4, we complete the proof of the main theorem.

\subsection*{Acknowledgements}
The author would like to thank Mathav Murugan for a discussion that led to the argument in Proposition~\ref{prop:plane-energy}, and Tianyi Zheng for helpful discussions and comments. This research is supported by JSPS KAKENHI Grant Number 25KJ0862 and FoPM, WINGS Program, the University of Tokyo.

\section{Preliminaries}\label{sec:preliminaries}
\subsection{Invariant point processes}
Throughout the paper, we fix $d\ge2$ and $\lambda>0$. Let $o$ be a base point of $X=\HH$, and write $d_{\mathbb{H}}$ and $\vol$ for hyperbolic distance and volume. We identify $B=\partial X$ with $\mathbb S^{d-1}$ in the Poincar\'e ball model centered at $o$, and write $m$ for normalized Lebesgue measure on this sphere. We write $G=\Isom^+(X)$ and $K=\operatorname{Stab}_G(o)$, so that $X=G/K$. On product spaces, we use the same notation $g$ for the diagonal action of $g\in G$. We first recall the basics of point processes on symmetric spaces, following \cite{AM22}.

\begin{definition}[Invariant point processes]
The configuration space is
\[
M(X)=\{\omega\subset X\colon\#(\omega\cap C)<\infty\text{ for every compact }C\subset X\}.
\]
We equip $M(X)$ with the Borel structure generated by the counting maps $\omega\mapsto\#(\omega\cap U)$ for bounded Borel sets $U\subset X$. This makes $M(X)$ a standard Borel space, on which $G$ acts by $g\omega=\{gx \colon x\in\omega\}$. A \emph{$G$-invariant point process} is a $G$-invariant Borel probability measure on $M(X)$.
\end{definition}

\begin{definition}[Poisson point process]
The \emph{Poisson point process} with intensity $\lambda\vol$ is characterized by the following two properties:
\begin{itemize}
\item For every Borel set $U\subset X$ of finite volume, the count $\#(\omega\cap U)$ has Poisson distribution with mean $\lambda\vol(U)$.
\item For every pair of disjoint Borel sets $U,V\subset X$, the counts $\#(\omega\cap U)$ and $\#(\omega\cap V)$ are independent.
\end{itemize}
We denote the process by $\Pi_\lambda$ and, since we fix the intensity, write $\mu$ for its law and $\EE_\mu$ for expectation under this law. Its law is $G$-invariant, and the action $G\curvearrowright(M(X),\mu)$ is ergodic by the independence property \cite{AM22}. This action can be seen as a continuous analogue of Bernoulli shifts.
\end{definition}

For $x\in\omega$, the \emph{Voronoi cell} of $x$ is
\[
C_x(\omega)=\bigcap_{y\in\omega}\{z\in X\colon d_{\mathbb{H}}(z,x)\le d_{\mathbb{H}}(z,y)\}.
\]
The \emph{Delaunay graph} $\mathcal G_\omega=(\omega,E_\omega)$ is the dual graph of the Voronoi tessellation, i.e.,
\begin{itemize}
\item Its vertex set is given by the configuration $\omega$;
\item There is an edge between distinct points $x,y\in\omega$ when their Voronoi cells share a facet.
\end{itemize}
For Poisson configurations, this is almost surely equivalent to requiring $C_x(\omega)\cap C_y(\omega)\ne\varnothing$. We write $x\sim_\omega y$ when $\{x,y\}\in E_\omega$, and $\deg_\omega(x)$ for the degree of $x$. Since the construction is $G$-equivariant, we have $\mathcal G_{g\omega}=g\mathcal G_\omega$.

\subsection{Palm measures and mass transport}\label{subsec:palm}
Although the event $o\in\omega$ has probability zero, the Palm measure gives a meaning to conditioning the point process to contain the base point $o$. Let
\[
M_o(X)=\{\omega\in M(X)\colon o\in\omega\},\,\, \Omega_o=K\backslash M_o(X).
\]
Since $K$ is compact, $\Omega_o$ is standard Borel. We do not distinguish between an element of a quotient space and its representative when no confusion can arise.

\begin{definition}[Palm measure]\label{def:palm}
Let $U\subset X$ be Borel with $0<\vol(U)<\infty$. The \emph{Palm measure} is the probability measure on $\Omega_o$ defined by
\[
\mu_o(A)=\frac{\EE_\mu\#\{gK\in\omega\cap U\colon g^{-1}\omega\in A\}}{\lambda\vol(U)}
\]
for each Borel subset $A \subset \Omega_o$.
\end{definition}

Note that the image of $g^{-1}\omega$ in $\Omega_o$ depends only on $gK$. For fixed $A$, the numerator, viewed as a function of $U$, defines a $G$-invariant locally finite measure on $X$. Since it is a constant multiple of hyperbolic volume, the definition is independent of $U$. We use the same notation $\mu_o$ for this measure and its $K$-invariant lift to $M_o(X)$. The latter is characterized by
\[
\int_{M_o(X)}F(\omega)\,\dd\mu_o(\omega)=\int_{\Omega_o}\int_K F(k\omega)\,\dd k\,\dd\mu_o(\omega)
\]
for every nonnegative Borel function $F$ on $M_o(X)$. Here $\dd k$ is Haar probability measure on $K$, and $\omega$ denotes any representative.

\begin{remark}\label{rem:slivnyak}
For the Poisson process, the Palm measure has the concrete description
\[
\mu_o=\operatorname{Law}(\Pi_\lambda\cup\{o\}).
\]
Indeed, for every nonnegative Borel $F$ on $M_o(X)$, the Mecke identity and invariance of the Poisson law give
\[
\EE_\mu\sum_{gK\in\omega\cap U}\int_K F(kg^{-1}\omega)\,\dd k=\lambda\vol(U)\EE_\mu F(\Pi_\lambda\cup\{o\}).
\]
Comparing with Definition~\ref{def:palm} identifies $\operatorname{Law}(\Pi_\lambda\cup\{o\})$ with $\mu_o$ as above \cite[Chapter~9]{LP18}.
\end{remark}

For every Borel $K$-invariant function $F\colon M_o(X)\to[0,\infty]$, Campbell's identity gives
\[
\EE_\mu\sum_{gK\in\omega\cap U}F(g^{-1}\omega)=\lambda\vol(U)\int_{M_o(X)}F(\omega)\,\dd\mu_o(\omega).
\]
In particular, a $G$-invariant Borel set is $\mu$-conull if and only if it is $\mu_o$-conull.

We use the following mass-transport principle; see \cite[Theorem~3.13 and Section~7]{AM22} and \cite[Corollary~6.2]{GL11}.
\begin{theorem}[Mass-transport principle]\label{thm:mass-transport}
Every diagonally $G$-invariant $\sigma$-finite Borel measure $M$ on $X\times X$ satisfies
\[
\int_{X\times X}\varphi(x)\,\dd M(x,y)=\int_{X\times X}\varphi(y)\,\dd M(x,y)
\]
for every nonnegative Borel function $\varphi$ on $X$. In particular, for a nonnegative Borel function $T(x,y;\omega)$, defined for $x,y\in\omega$ and satisfying $T(gx,gy;g\omega)=T(x,y;\omega)$, we have
\[
\EE_{\mu_o}\sum_{x\in\omega}T(o,x;\omega)=\EE_{\mu_o}\sum_{x\in\omega}T(x,o;\omega).
\]
\end{theorem}

By standard properties of Poisson--Voronoi tessellations \cite{BKP21} and the multivariate Mecke identity, there is a $G$-invariant $\mu$-conull Borel set on which all Voronoi cells are compact, $\mathcal G_\omega$ is connected and locally finite, and the stabilizer of $\omega$ in $G$ is trivial. By Campbell's identity, the same properties hold for $\mu_o$-almost every $\omega$. We restrict to this conull set and its intersection with $M_o(X)$ throughout the paper.

\subsection{Rerooting relation}
Following \cite[Section~7]{AM22}, we define the \emph{rerooting relation} by
\[
\cR=\{(\omega,g^{-1}\omega)\in\Omega_o^2\colon gK\in\omega\},
\]
which is a countable Borel equivalence relation on $\Omega_o$. Since the configurations have trivial stabilizers, the map $gK\mapsto g^{-1}\omega$ identifies the vertices of $\mathcal G_\omega$ with the $\cR$-class of $\omega$.

Recall that a partial isomorphism of a countable Borel equivalence relation on a standard Borel space is a Borel bijection between Borel subsets whose graph is contained in the relation. The relation is called \emph{measure-preserving} if every such map preserves the restrictions of the measure to its domain and range. It is called \emph{nonsingular} if these maps preserve null sets, and \emph{ergodic} if every invariant Borel set is null or conull.

\begin{proposition}\label{prop:rerooting}
The relation $\cR$ is measure-preserving and ergodic with respect to $\mu_o$.
\end{proposition}

\begin{proof}
Measure preservation follows from the mass-transport principle, see \cite[Section~7]{AM22} for details.

Let $A\subset\Omega_o$ be $\cR$-invariant with $\mu_o(A)>0$, and let $\widetilde A\subset M(X)$ be the Borel set of configurations admitting a rerooting in $A$. Since invariance of $A$ implies that every rerooting of such a configuration belongs to $A$, the set $\widetilde A$ is $G$-invariant. For a bounded Borel $U\subset X$ of positive volume, the definition of the Palm measure gives
\[
\lambda\vol(U)\mu_o(A)=\EE_\mu\bigl[\one_{\widetilde A}(\omega)\#(\omega\cap U)\bigr].
\]
Thus $\mu(\widetilde A)>0$, and ergodicity gives $\mu(\widetilde A)=1$. The same identity now yields $\mu_o(A)=1$.
\end{proof}

For a Borel set $N\subset\Omega_o$, let $[N]_{\cR}$ denote its saturation under $\cR$, which is Borel and is null when $N$ is null.

\subsection{Random walks and harmonic measures}\label{subsec:bilateral}
Although Theorem~\ref{thm:main} concerns the usual simple random walk, it is useful to work with a bilateral version in the proof.

For a configuration $\omega$ and $x\in\omega$, let $\PP_x^\omega$ be the law on $X^\ZZ$ of a bilateral simple random walk: the future $(X_n)_{n\ge0}$ and the past $(X_{-n})_{n\ge0}$ are independent simple random walks with $X_0=x$ and transition probabilities
\[
p_\omega(x,y)=\frac{\one_{\{x\sim_\omega y\}}}{\deg_\omega(x)}.
\]
The map $(\omega,x)\mapsto\PP_x^\omega$ is Borel from $\{(\omega,x)\in M(X)\times X\colon x\in\omega\}$ to the space of Borel probability measures on $X^\ZZ$. It is also $G$-equivariant, that is, $\PP_{gx}^{g\omega}=g_*\PP_x^\omega$, where $G$ acts diagonally on $X^\ZZ$.

\begin{lemma}\label{lem:moments}
We have
\[
D_\lambda:=\EE_{\mu_o}\deg_\omega(o)<\infty,\,\, \ell_j:=\EE_{\mu_o}\sum_{y\sim_\omega o}d_{\mathbb{H}}(o,y)^j<\infty
\]
for $j=1,2$.
\end{lemma}
\begin{proof}
The estimates in \cite[Section~3]{Paq18} give finite moments of every order for $\operatorname{diam}C_o(\omega)$ and $\deg_\omega(o)$. Since $d_{\mathbb{H}}(o,y)\le2\operatorname{diam}C_o(\omega)$ for every $y\sim_\omega o$, we have
\[
\ell_j\le 2^j\bigl(\EE_{\mu_o}\deg_\omega(o)^2\bigr)^{1/2}\bigl(\EE_{\mu_o}[(\operatorname{diam}C_o(\omega))^{2j}]\bigr)^{1/2}<\infty
\]
for $j=1,2$.
\end{proof}

Let $\widehat\mu_o$ be the degree-biased version of $\mu_o$ defined by
\[
\dd\widehat\mu_o(\omega)=\frac{\deg_\omega(o)}{D_\lambda}\,\dd\mu_o(\omega),
\]
which is equivalent to $\mu_o$ by Lemma~\ref{lem:moments} and is stationary for the environment seen from the walk, as shown below.

\begin{definition}
The \emph{path bundle} is the standard Borel space
\[
\mathcal W=K\backslash\bigl\{(\omega,(X_n)_{n\in\ZZ})\colon\omega\in M_o(X),\ X_0=o,\ X_n\sim_\omega X_{n+1}\text{ for all }n\in\ZZ\bigr\},
\]
where $K$ acts diagonally. We equip it with the probability measure
\[
\Lambda=\int_{\Omega_o}\delta_\omega\otimes\PP_o^\omega\,\dd\widehat\mu_o(\omega).
\]
In the integrand, $\delta_\omega\otimes\PP_o^\omega$ is viewed as a measure on $\mathcal W$ by taking the quotient by $K$. We use the same convention for measures on quotient spaces below. The \emph{bilateral shift} $T\colon\mathcal W\to\mathcal W$ is given by
\[
T\bigl(\omega,(X_n)_{n\in\ZZ}\bigr)=\bigl(g^{-1}\omega,(g^{-1}X_{n+1})_{n\in\ZZ}\bigr),
\]
where $g\in G$ satisfies $go=X_1$. Note that this is well defined by taking the left quotient by $K$.
\end{definition}

The following proposition follows by standard arguments; see, for example, \cite[Theorems~4.1 and~4.6]{AL07}.
\begin{proposition}\label{prop:shift}
The system $(\mathcal W,\Lambda,T)$ is invertible, measure-preserving, and ergodic.
\end{proposition}
\begin{proof}
The inverse of $T$ shifts the path one step backward and moves $X_{-1}$ to $o$. For nonnegative Borel functions $a,b$ on $\Omega_o$, the mass-transport principle gives
\[
\EE_{\mu_o}\sum_{gK\sim_\omega o}a(\omega)b(g^{-1}\omega)
=\EE_{\mu_o}\sum_{gK\sim_\omega o}b(\omega)a(g^{-1}\omega).
\]
The factor $\deg_\omega(o)$ in $\widehat\mu_o$ cancels the denominator in the transition probabilities, so the chain of rooted configurations is reversible under $\widehat\mu_o$. Applying the same argument to finite cylinders shows that $T$ preserves $\Lambda$.

For a Borel set $C\subset\mathcal W$, let
\[
f_C(\omega)=\PP_o^\omega(\{w\in X^\ZZ\colon(\omega,w)\in C\}).
\]
Let $A\subset\mathcal W$ be $T$-invariant. We first show that $f_A(\omega)\in\{0,1\}$ for $\widehat\mu_o$-almost every $\omega$. For $\varepsilon>0$, there are $N\ge1$ and a Borel set $A_N$ depending only on $(\omega,(X_n)_{-N\le n\le N})$ such that $\Lambda(A\mathbin\triangle A_N)<\varepsilon$. The events $T^NA_N$ and $T^{-N}A_N$ depend respectively on the past and the future. Therefore, conditional independence given $\omega$ gives
\[
f_{T^NA_N\cap T^{-N}A_N}=f_{T^NA_N}f_{T^{-N}A_N}
\]
for $\widehat\mu_o$-almost every $\omega$. Since $T$ preserves $\Lambda$ and $A$ is $T$-invariant, we obtain
\begin{align*}
\|f_A-f_A^2\|_1
&\le\|f_A-f_{T^NA_N}f_{T^{-N}A_N}\|_1+2\varepsilon\\
&=\|f_A-f_{T^NA_N\cap T^{-N}A_N}\|_1+2\varepsilon\\
&\le\Lambda\bigl(A\mathbin\triangle(T^NA_N\cap T^{-N}A_N)\bigr)+2\varepsilon<4\varepsilon,
\end{align*}
where $\|\cdot\|_1$ is the norm on $L^1(\Omega_o,\widehat\mu_o)$. Letting $\varepsilon\to0$ proves that $f_A(\omega)\in\{0,1\}$ for $\widehat\mu_o$-almost every $\omega$. Hence $\one_A(\omega,w)=f_A(\omega)$ for $\Lambda$-almost every $(\omega,w)$.

The $T$-invariance of $A$ now gives
\[
f_A(\omega)=\frac1{\deg_\omega(o)}\sum_{gK\sim_\omega o}f_A(g^{-1}\omega)
\]
for $\widehat\mu_o$-almost every $\omega$. Since $f_A$ takes values in $\{0,1\}$, its value is unchanged at every neighboring root. By removing the $\cR$-saturation of the exceptional null set, we obtain that $f_A$ is $\cR$-invariant. By Proposition~\ref{prop:rerooting}, $f_A$ is constant for $\widehat\mu_o$-almost every $\omega$, which implies $\Lambda(A)\in\{0,1\}$.
\end{proof}

By \cite{BKP21}, the simple random walk has positive hyperbolic speed. More precisely, there exists $\ell>0$ such that, for $\mu_o$-almost every $\omega$ and every $x\in\omega$, we have
\[
\frac1n d_{\mathbb{H}}(o,X_{\pm n})\to\ell
\]
$\PP_x^\omega$-almost surely. The moment bound above, stationarity, and the Borel--Cantelli lemma allow us to apply Kaimanovich's regularity criterion \cite[Theorem~7.2]{Kai00}, which yields that almost every bilateral sample path determines unique points $\Xi_+,\Xi_-\in B$ such that
\[
\frac1n d_{\mathbb{H}}\bigl(X_{\pm n},\Xi_\pm(\ell n)\bigr)\to0,
\]
where each $\xi\in B$ is identified with the unit-speed geodesic ray from $o$ ending at $\xi$.

The \emph{boundary maps} are given by
\[
\operatorname{bnd}_\pm((X_n)_{n\in\ZZ})=\lim_{n\to\infty}X_{\pm n}=\Xi_\pm.
\]
For $\omega\in M(X)$ and $x\in\omega$, we define the \emph{harmonic measure} $\nu_{\omega,x}$ by
\[
\nu_{\omega,x}=(\operatorname{bnd}_+)_*\PP_x^\omega=(\operatorname{bnd}_-)_*\PP_x^\omega.
\]
The harmonic measures are measurable by construction. For $g\in G$ and adjacent $x,y\in\omega$, equivariance and the first-step decomposition give
\[
\nu_{g\omega,gx}=g_*\nu_{\omega,x},\,\, \nu_{\omega,x}\ge\deg_\omega(x)^{-1}\nu_{\omega,y}.
\]
Interchanging $x,y$ shows that the measure class is independent of the starting vertex. In particular,
\[
\nu_{\omega,o}\sim\nu_{\omega,x}
\]
for every $\omega\in M_o(X)$ and $x\in\omega$.

Next we introduce the (double) boundary bundle and the extended rerooting relation on it, which play crucial roles in our proof.

\begin{definition}
The \emph{boundary bundle} is the standard Borel space
\[
\mathcal B=K\backslash\bigl(M_o(X)\times B^2\bigr),
\]
equipped with the probability measure
\[
Q=\int_{\Omega_o}\delta_\omega\otimes\nu_{\omega,o}\otimes\nu_{\omega,o}\,\dd\mu_o(\omega).
\]
The \emph{boundary extension} of the rerooting relation is given by
\[
\widehat\cR=\bigl\{\bigl((\omega,\xi,\eta),(g^{-1}\omega,g^{-1}\xi,g^{-1}\eta)\bigr)\in\mathcal B^2\colon gK\in\omega\bigr\},
\]
which is a countable Borel equivalence relation on $\mathcal B$.
\end{definition}

\begin{corollary}[Double ergodicity]\label{cor:double-ergodicity}
The boundary extension $\widehat\cR$ is nonsingular and ergodic with respect to $Q$.
\end{corollary}
\begin{proof}
Let $\phi\colon A\to D$ be a partial isomorphism of $\cR$. Its lift to the boundary bundle is
\[
\widehat\phi(\omega,\xi,\eta)=(g^{-1}\omega,g^{-1}\xi,g^{-1}\eta),
\]
where $\phi(\omega)=g^{-1}\omega$. Equivariance of harmonic measure and equivalence under a change of root give
\[
(g^{-1})_*\nu_{\omega,o}=\nu_{g^{-1}\omega,g^{-1}o}\sim\nu_{g^{-1}\omega,o}.
\]
Thus the lift preserves the measure class of the conditional measures $\nu_{\omega,o}^{\otimes2}$. Since $\phi$ preserves $\mu_o$, both $\widehat\phi$ and its inverse preserve $Q$-null sets. Since these lifts generate $\widehat\cR$, the relation $\widehat\cR$ is nonsingular with respect to $Q$.

Let $b\colon \mathcal W\to\mathcal B$ be the map given by
\[
b\bigl(\omega,(X_n)_{n\in\ZZ}\bigr)=(\omega,\Xi_-,\Xi_+),
\]
which sends every $T$-orbit into a $\widehat\cR$-class. Conditional independence of the past and future and the definition of $\widehat\mu_o$ give
\[
\frac{\dd(b_*\Lambda)}{\dd Q}(\omega,\xi,\eta)=\frac{\deg_\omega(o)}{D_\lambda},
\]
and hence $b_*\Lambda\sim Q$.
If $A\subset\mathcal B$ is $\widehat\cR$-invariant, then $b^{-1}(A)$ is $T$-invariant, and hence $\Lambda(b^{-1}(A))\in\{0,1\}$ by Proposition~\ref{prop:shift}. Therefore, $Q(A)\in\{0,1\}$ by $b_*\Lambda\sim Q$.
\end{proof}

For a measure $\nu$ and a measurable function $f$, we write $\nu(f)=\int f\,\dd\nu$ whenever the integral is defined. For a bounded Borel function $f\colon B\to\RR$, we define its \emph{harmonic extension} $H_\omega f$ on $\mathcal G_\omega$ by
\[
H_\omega f(x)=\nu_{\omega,x}(f)
\]
for $x\in\omega$. The Markov property shows that $H_\omega f$ is harmonic and that
\[
H_\omega f(X_n)=\EE_x^\omega[f(\Xi_+)\mid X_0,\ldots,X_n]\to f(\Xi_+)
\]
$\PP_x^\omega$-almost surely, by the martingale convergence theorem.

\subsection{Dichotomy for measure classes}
To compare harmonic measure with Lebesgue measure, we define another probability measure on $\mathcal B$ by
\[
Q_m=\mu_o\otimes m^{\otimes2}.
\]
The relation $\widehat\cR$ is nonsingular with respect to $Q_m$ as well. Indeed, its base maps preserve $\mu_o$, and the $G$-action preserves the measure class of $m$.

\begin{samepage}
\begin{lemma}\label{lem:pure-type}
For the Poisson point process on $X=\HH$, one of the following holds:
\begin{enumerate}
\item For $\mu_o$-almost every $\omega$ and every $x\in\omega$, we have $\nu_{\omega,x}\ll m$.
\item For $\mu_o$-almost every $\omega$ and every $x\in\omega$, we have $\nu_{\omega,x}\perp m$.
\end{enumerate}
\end{lemma}
\end{samepage}
\begin{proof}
Let $Q=Q_{\mathrm a}+Q_{\mathrm s}$ be the Lebesgue decomposition with respect to $Q_m$, and let $N\subset\mathcal B$ be Borel with $Q_m(N)=0$ and $Q_{\mathrm s}(\mathcal B\setminus N)=0$. Since $\widehat\cR$ is countable and nonsingular for $Q_m$, its saturation $[N]_{\widehat\cR}$ is $Q_m$-null. This saturation is a union of $\widehat\cR$-classes and is therefore $\widehat\cR$-invariant. By Corollary~\ref{cor:double-ergodicity}, it is $Q$-null or $Q$-conull. In the first case, $Q_{\mathrm s}=0$, and then $Q\ll Q_m$. In the second case, $Q$ is supported on a $Q_m$-null set, and then $Q\perp Q_m$.

Suppose first that $Q\ll Q_m$, and let $q=\dd Q/\dd Q_m$. Lifting to $M_o(X)\times B^2$ and disintegrating over $M_o(X)$ gives
\[
\nu_{\omega,o}^{\otimes2}=q(\omega,\cdot,\cdot)\,m^{\otimes2}
\]
for $\mu_o$-almost every $\omega$. Taking the first marginal of each side gives $\nu_{\omega,o}\ll m$ for $\mu_o$-almost every $\omega$. If $Q\perp Q_m$, let $N$ be a Borel set with $Q(N)=1$ and $Q_m(N)=0$, and let $N_\omega=\{(\xi,\eta)\in B^2\colon (\omega,\xi,\eta)\in N\}$. Disintegrating the lifted measures gives, for $\mu_o$-almost every $\omega$,
\[
\nu_{\omega,o}^{\otimes2}(N_\omega)=1,\,\, m^{\otimes2}(N_\omega)=0.
\]
Thus $\nu_{\omega,o}^{\otimes2}\perp m^{\otimes2}$ for $\mu_o$-almost every $\omega$. Suppose, for a contradiction, that $\nu_{\omega,o}$ has a nonzero absolutely continuous part $\alpha$. Then we have $0\ne\alpha\otimes\alpha\le\nu_{\omega,o}^{\otimes2}$ and $\alpha\otimes\alpha\ll m^{\otimes2}$, contradicting $\nu_{\omega,o}^{\otimes2}\perp m^{\otimes2}$. Hence $\nu_{\omega,o}\perp m$ for $\mu_o$-almost every $\omega$. Since $\nu_{\omega,x}\sim\nu_{\omega,o}$ for every $x\in\omega$, absolute continuity or singularity holds at every vertex.
\end{proof}

\section{Currents on the boundary}\label{sec:current}
Let $B^{(2)}=\{(\xi,\eta)\in B^2\colon \xi\ne\eta\}$ be the space of pairs of distinct boundary points. The \emph{Liouville current} on $B^{(2)}$ is given by
\[
\dd L(\xi,\eta)=\frac{\dd m(\xi)\,\dd m(\eta)}{|\xi-\eta|^{2(d-1)}}.
\]
It is straightforward to check that $L$ is invariant under the diagonal $G$-action, Radon, and equivalent to $m^{\otimes2}|_{B^{(2)}}$.

\subsection{Na\"im currents}\label{subsec:naim}
We first construct the Na\"im current for a fixed graph $\mathcal G_\omega$. Let $W_\omega$ be the space of bilateral paths in $\mathcal G_\omega$ which visit each vertex only finitely often and converge to $B$ in both directions. Let $\theta\colon W_\omega\to W_\omega$ be the time shift defined by
\[
\theta((X_n)_{n\in\ZZ})=(X_{n+1})_{n\in\ZZ},
\]
and let $P_\omega$ be the measure on $W_\omega$ defined by
\[
P_\omega=\sum_{x\in\omega}\deg_\omega(x)\PP_x^\omega.
\]
Note that $\theta$ changes only the time origin and leaves $\omega$ fixed, unlike the rerooting shift $T$. For adjacent $x,y\in\omega$, the detailed-balance identity
\[
\deg_\omega(x)p_\omega(x,y)=\deg_\omega(y)p_\omega(y,x)=1
\]
shows that $P_\omega$ is invariant under $\theta$ and time reversal.

By transience, $d_{\mathbb{H}}(o,X_n)$ tends to infinity as $n\to\pm\infty$ and attains its minimum at only finitely many times. Let $\mathcal D_\omega$ consist of the paths for which this minimum is first attained at time zero, giving a Borel fundamental domain for the shift. Let $q_\omega\colon W_\omega\to W_\omega^*=W_\omega/\langle\theta\rangle$ be the factor map corresponding to the $\theta$-invariant Borel $\sigma$-algebra. The fundamental domain identifies $W_\omega^*$ with a standard Borel space. The measure
\[
I_\omega=(q_\omega)_*(P_\omega|_{\mathcal D_\omega})
\]
is $\sigma$-finite. If $\mathcal D'_\omega$ is another Borel fundamental domain, partitioning $\mathcal D_\omega$ according to the shift taking its points into $\mathcal D'_\omega$ and using $\theta$-invariance of $P_\omega$ gives
\[
P_\omega(\mathcal D_\omega\cap q_\omega^{-1}(A))=P_\omega(\mathcal D'_\omega\cap q_\omega^{-1}(A))
\]
for every Borel set $A\subset W_\omega^*$, i.e., $I_\omega$ does not depend on the fundamental domain. The measure $I_\omega$ is preserved by time reversal, since time reversal preserves $P_\omega$ and takes a fundamental domain to another one. Let $J_\omega$ be the image of $I_\omega$ on $B^2$ under the boundary maps:
\[
J_\omega=(\operatorname{bnd}_-,\operatorname{bnd}_+)_*I_\omega.
\]

For a finite set $F\subset\omega$, let
\[
W_{\omega,F}^*=\{q_\omega(w)\colon w\in W_\omega,\ X_n\in F\text{ for some }n\in\ZZ\},
\]
and let
\[
J_{\omega,F}=(\operatorname{bnd}_-,\operatorname{bnd}_+)_*(I_\omega|_{W_{\omega,F}^*}).
\]
For the one-sided walk, let
$\tau_F=\inf\{n\ge0\colon X_n\in F\}$ and $\tau_F^+=\inf\{n\ge1\colon X_n\in F\}$.
For $x\in F$ and Borel $A\subset B$, let
\[
\rho_{\omega,F,x}(A)=\deg_\omega(x)\PP_x^\omega(\Xi_+\in A,\tau_F^+=\infty).
\]

\begin{lemma}\label{lem:current-construction}
For every finite $F\subset\omega$, we have
\[
J_{\omega,F}=\sum_{x\in F}\rho_{\omega,F,x}\otimes\nu_{\omega,x}.
\]
The measures $J_{\omega,F}$ are finite and symmetric. For every increasing exhaustion $(F_n)_{n\ge1}$ of $\omega$ by finite subsets,
\[
J_{\omega,F_n}(A)\uparrow J_\omega(A)
\]
for every Borel set $A\subset B^2$. Moreover, $J_\omega\sim\nu_{\omega,x}^{\otimes2}$ for every $x\in\omega$.
\end{lemma}
\begin{proof}
For $w^*\in W_{\omega,F}^*$, transience allows us to choose a representative $(X_n)_{n\in\ZZ}$ such that
\[
X_0\in F,\,\, X_n\notin F\,\,\text{for every }n<0.
\]
Given $X_0=x\in F$, the past is conditioned to avoid $F$ before time zero, while the future is an independent walk from $x$. Thus, for Borel sets $A,C\subset B$, we have
\[
P_\omega(X_0=x,\ \{X_{-n}\colon n\ge1\}\cap F=\varnothing,\ \Xi_-\in A,\ \Xi_+\in C)
=\rho_{\omega,F,x}(A)\nu_{\omega,x}(C).
\]
Summing over $x\in F$ gives the description of $J_{\omega,F}$, which satisfies
\[
J_{\omega,F}(B^2)=\sum_{x\in F}\rho_{\omega,F,x}(B)\le\sum_{x\in F}\deg_\omega(x)<\infty.
\]
Further, since time reversal preserves $I_\omega$ and $W_{\omega,F}^*$, the measure $J_{\omega,F}$ is symmetric. For every increasing exhaustion $(F_n)_{n\ge1}$ of $\omega$, we have $W_{\omega,F_n}^*\uparrow W_\omega^*$. Hence $J_{\omega,F_n}(A)\uparrow J_\omega(A)$ for every Borel set $A\subset B^2$.

It remains to show that $J_\omega$ is equivalent to $\nu_{\omega,x}^{\otimes2}$ for every $x\in\omega$. Since $\rho_{\omega,F,y}\le\deg_\omega(y)\nu_{\omega,y}$ and $\nu_{\omega,y}\sim\nu_{\omega,x}$, the description of $J_{\omega,F}$ gives $J_{\omega,F}\ll\nu_{\omega,x}^{\otimes2}$, which gives $J_\omega\ll\nu_{\omega,x}^{\otimes2}$ as $F$ increases to $\omega$. For the reverse direction, it suffices to consider $F=\{x\}$: indeed, the strong Markov property gives
\[
\nu_{\omega,x}(A)=\PP_x^\omega(\Xi_+\in A,\tau_{\{x\}}^+=\infty)+\PP_x^\omega(\tau_{\{x\}}^+<\infty)\nu_{\omega,x}(A)
\]
for every Borel subset $A\subset B$, and hence
\[
J_{\omega,\{x\}}=\deg_\omega(x)\PP_x^\omega(\tau_{\{x\}}^+=\infty)\nu_{\omega,x}^{\otimes2}\le J_\omega,
\]
which completes the proof by transience.
\end{proof}

We call the restriction of $J_\omega$ to $B^{(2)}$ the \emph{Na\"im current} of $\mathcal G_\omega$, and also denote it by $J_\omega$.

Note that the maps $\omega\mapsto I_\omega$ and $\omega\mapsto J_\omega$ are measurable by construction. For $g\in G$, we have $P_{g\omega}=g_*P_\omega$, and $g\mathcal D_\omega$ is a fundamental domain for $W_{g\omega}$. Therefore, independence of the fundamental domain and $G$-equivariance of the boundary maps give
\[
I_{g\omega}=g_*I_\omega,\,\, J_{g\omega}=g_*J_\omega.
\]

To apply the Doob–Naïm formula in Subsection 4.2, we identify $J_\omega$ with the pushforward of the Naïm current on the Martin boundary as in Silverstein's work \cite{Sil74}. For $a\in\omega$, let $\widetilde B_\omega$ be the minimal Martin boundary, with hitting measures $\widetilde\nu_{\omega,x}$ and limits $\widetilde\Xi_\pm$. For a finite set $F\subset\omega$, define the finite current $\widetilde J_{\omega,F}$ and measures $\widetilde\rho_{\omega,F,x}$ on $\widetilde B_\omega$ by
\[
\widetilde J_{\omega,F}=(\widetilde\Xi_-,\widetilde\Xi_+)_*(I_\omega|_{W_{\omega,F}^*}),\,\,
\widetilde\rho_{\omega,F,x}(A)=\deg_\omega(x)\PP_x^\omega(\widetilde\Xi_+\in A,\tau_F^+=\infty).
\]
The argument in Lemma~\ref{lem:current-construction} also applies to the Martin limits and gives
\[
\widetilde J_{\omega,F}=\sum_{x\in F}\widetilde\rho_{\omega,F,x}\otimes\widetilde\nu_{\omega,x}.
\]
As $F$ increases to $\omega$, these finite currents increase to $(\widetilde\Xi_-,\widetilde\Xi_+)_*I_\omega$, which Silverstein \cite[Section~1]{Sil74} identified as
\[
(\widetilde\Xi_-,\widetilde\Xi_+)_*I_\omega=\deg_\omega(a)\Theta_\omega\,\widetilde\nu_{\omega,a}^{\otimes2},
\]
where $\Theta_\omega$ is the Na\"im kernel with reference vertex $a$. The prefactor is $\deg_\omega(a)$ because Silverstein's symmetrizing measure $\alpha$ is $\alpha(x)=\deg_\omega(x)$ for this walk.

The geometric boundary $B=\partial X$ is a measurable factor of $\widetilde B_\omega$: for every $f\in C(B)$, there is a bounded measurable function $\widetilde f$ on $\widetilde B_\omega$ such that
\[
H_\omega f(x)=\int_{\widetilde B_\omega}\widetilde f(\zeta)\,\dd\widetilde\nu_{\omega,x}(\zeta),\,\, H_\omega f(X_n)\to\widetilde f(\widetilde\Xi_+)
\]
$\PP_x^\omega$-almost surely for every $x\in\omega$ \cite[Section~1]{Sil74}. Since the martingale convergence theorem also gives $H_\omega f(X_n)\to f(\Xi_+)$, we have
\[
\widetilde f(\widetilde\Xi_+)=f(\Xi_+)
\]
$\PP_x^\omega$-almost surely. Taking a countable family of continuous functions separating the points of $B$ gives a measurable map $\pi_\omega\colon\widetilde B_\omega\to B$ such that
\[
\Xi_\pm=\pi_\omega(\widetilde\Xi_\pm),\,\, \nu_{\omega,x}=(\pi_\omega)_*\widetilde\nu_{\omega,x}
\]
with the first equality holding $\PP_x^\omega$-almost surely for every $x\in\omega$. The same applies to the reversed walk, and equivalence of the hitting measures makes it valid for all starting vertices. Pushing Silverstein's identity forward by $\pi_\omega\times\pi_\omega$ and restricting to $B^{(2)}$ gives
\[
J_\omega=\left.(\pi_\omega\times\pi_\omega)_*\bigl(\deg_\omega(a)\Theta_\omega\,\widetilde\nu_{\omega,a}^{\otimes2}\bigr)\right|_{B^{(2)}}.
\]

To prove that $J_\omega$ is almost surely Radon, we estimate its mass on pairs of boundary points joined by geodesics meeting a bounded region. For $(\xi,\eta)\in B^{(2)}$, let $\gamma_{\xi,\eta}$ be the geodesic joining $\xi$ and $\eta$. For $R\ge0$, define the compact sets
\[
C_R=\{(\xi,\eta)\in B^{(2)}\colon d_{\mathbb{H}}(o,\gamma_{\xi,\eta})\le R\},
\]
which exhaust $B^{(2)}$.

\begin{proposition}\label{prop:local-finiteness}
For $\mu_o$-almost every $\omega$, the current $J_\omega$ is Radon on $B^{(2)}$.
\end{proposition}
\begin{proof}
Let $\ell_{\mathbb{H}}$ denote hyperbolic arclength. Every geodesic meeting $\overline B(o,R)$ contains a segment of length two in $\overline B(o,R+1)$. Therefore,
\[
2J_\omega(C_R)\le\int_{B^{(2)}}\ell_{\mathbb H}(\gamma_{\xi,\eta}\cap\overline B(o,R+1))\,\dd J_\omega(\xi,\eta).
\]
Thus it is enough to show that
\[
\EE_\mu\int_{B^{(2)}}\ell_{\mathbb H}(\gamma_{\xi,\eta}\cap U)\,\dd J_\omega(\xi,\eta)\le\lambda\ell_1\vol(U)
\]
for every ball $U\subset X$, where $\ell_1$ is defined in Lemma~\ref{lem:moments}.

For a trajectory $w=(w_n)_{n\in\ZZ}$ with distinct boundary limits, let $\gamma_w$ be the geodesic determined by these limits and let $\pi_w$ be nearest-point projection onto $\gamma_w$. For $n \in \ZZ$ we define the geodesic segments
\[
I_{w,n}=[\pi_w(w_n),\pi_w(w_{n+1})],
\]
which satisfy
\[
\ell_{\mathbb H}(I_{w,n})\le d_{\mathbb{H}}(w_n,w_{n+1}),\,\,
\ell_{\mathbb H}(\gamma_w\cap U)\le\sum_{n\in\ZZ}\ell_{\mathbb H}(I_{w,n}\cap U),
\]
since the projected points tend to opposite ends of $\gamma_w$.

For every edge $x\sim_\omega y$, the definition of $I_\omega$ gives
\[
\int_{W_\omega^*}\sum_{n\in\ZZ}\one_{\{w_n=x,\ w_{n+1}=y\}}\,\dd I_\omega(w^*)
=P_\omega(X_0=x,X_1=y)=1.
\]
Let $W_\omega^{*,(2)}$ be the set of trajectories modulo time shift with distinct boundary limits. We also define a measure on $X\times X$ by
\[
M=\EE_\mu\int_{W_\omega^{*,(2)}}\sum_{n\in\ZZ}\delta_{w_n}\otimes\bigl(\ell_{\mathbb H}|_{I_{w,n}}\bigr)\,\dd I_\omega(w^*),
\]
where $w$ is any representative of $w^*$, and the sum is independent of this choice. By the construction, it is clear that this measure $M$ is invariant under the diagonal $G$-action. Further, Campbell's identity gives
\begin{align*}
M(U\times X)
&=\EE_\mu\int_{W_\omega^{*,(2)}}\sum_{n\in\ZZ}\one_U(w_n)\ell_{\mathbb H}(I_{w,n})\,\dd I_\omega(w^*)\\
&\le\EE_\mu\int_{W_\omega^*}\sum_{n\in\ZZ}\one_U(w_n)d_{\mathbb{H}}(w_n,w_{n+1})\,\dd I_\omega(w^*)\\
&=\EE_\mu\sum_{x\in\omega\cap U}\sum_{y\sim_\omega x}d_{\mathbb{H}}(x,y)
=\lambda\ell_1\vol(U)<\infty,
\end{align*}
and in particular $M$ is $\sigma$-finite. Then we can apply the mass-transport principle to $\varphi=\one_U$ and obtain
\begin{align*}
\EE_\mu\int_{B^{(2)}}\ell_{\mathbb H}(\gamma_{\xi,\eta}\cap U)\,\dd J_\omega(\xi,\eta)
&\le\EE_\mu\int_{W_\omega^{*,(2)}}\sum_{n\in\ZZ}\ell_{\mathbb H}(I_{w,n}\cap U)\,\dd I_\omega(w^*)\\
&=M(X\times U)=M(U\times X)\\
&\le\lambda\ell_1\vol(U),
\end{align*}
which completes the proof.
\end{proof}

\subsection{Rigidity of the Na\"im current}\label{subsec:current-density}
We now establish rigidity of the Na\"im current under the assumption of absolute continuity.

\begin{proposition}\label{prop:rigidity}
Suppose that $\nu_{\omega,o}\ll m$ for $\mu_o$-almost every $\omega$. Then there are a deterministic constant $c_{d,\lambda}\in(0,\infty)$ and a measurable family of Borel sets $A_\omega\subset B^{(2)}$ such that
\[
J_\omega=c_{d,\lambda}L|_{A_\omega}
\]
for $\mu_o$-almost every $\omega$.
\end{proposition}

\begin{proof}
Absolute continuity implies that $Q$ gives no mass to the diagonal, and hence we can restrict the boundary bundle to $K\backslash(M_o(X)\times B^{(2)})$, retaining the notation $\mathcal B$ and $Q$. On this bundle, we define
\begin{itemize}
    \item $\displaystyle \mathbf J=\int_{\Omega_o}\delta_\omega\otimes J_\omega\,\dd\mu_o(\omega)$,
    \item $\displaystyle \mathbf L=\mu_o\otimes L$.
\end{itemize}
Proposition~\ref{prop:local-finiteness} shows that both measures are $\sigma$-finite. By Lemma~\ref{lem:current-construction}, we have $\mathbf J\ll\mathbf L$ and $\mathbf J\sim Q$. Let $\Phi=\dd\mathbf J/\dd\mathbf L$ be the Radon--Nikodym derivative, and use the same notation for its lift to $M_o(X)\times B^{(2)}$. Disintegration gives
\[
\Phi(\omega,\xi,\eta)=\frac{\dd J_\omega}{\dd L}(\xi,\eta)
\]
for $\mu_o$-almost every $\omega$ and $L$-almost every $(\xi,\eta)\in B^{(2)}$. Since $\mathbf J\sim Q$, we have $0<\Phi<\infty$, $Q$-a.e.

We show that $\Phi$ is $\widehat\cR$-invariant. Let $\phi$ be a partial isomorphism of $\cR$, with $\phi(\omega)=g^{-1}\omega$, and let $\widehat\phi$ be its lift as in the proof of Corollary~\ref{cor:double-ergodicity}. Since $\phi$ preserves $\mu_o$ and $L$ is $G$-invariant, $\widehat\phi$ preserves $\mathbf L$. Combining the $G$-equivariance of $J_\omega$ and the $G$-invariance of $L$, we obtain
\begin{align*}
\Phi(g^{-1}\omega,g^{-1}\xi,g^{-1}\eta)
&=\frac{\dd J_{g^{-1}\omega}}{\dd L}(g^{-1}\xi,g^{-1}\eta)\\
&=\frac{\dd (g^{-1})_*J_\omega}{\dd (g^{-1})_*L}(g^{-1}\xi,g^{-1}\eta)\\
&=\frac{\dd J_\omega}{\dd L}(\xi,\eta)
=\Phi(\omega,\xi,\eta)
\end{align*}
for $\mathbf L$-almost every $(\omega,\xi,\eta)$ in the domain of $\widehat\phi$. Since the lifts generate $\widehat\cR$, and $Q\ll\mathbf L$, we obtain that $\Phi$ is $\widehat\cR$-invariant $Q$-a.e. Therefore, Corollary~\ref{cor:double-ergodicity} gives a constant $c_{d,\lambda}\in(0,\infty)$ such that $\Phi=c_{d,\lambda}$, $Q$-a.e.

Set $A=\{\Phi>0\}$ and $A_\omega=\{(\xi,\eta)\in B^{(2)}:\Phi(\omega,\xi,\eta)>0\}$. Since $\mathbf J=\Phi\mathbf L$ and $\mathbf J\sim Q$, we have $\mathbf L|_A\sim Q$. Hence $\Phi=c_{d,\lambda}\one_A$, $\mathbf L$-a.e., which implies
\[
J_\omega=\Phi(\omega,\cdot,\cdot)L=c_{d,\lambda}L|_{A_\omega}
\]
for $\mu_o$-almost every $\omega$.
\end{proof}

\section{Proof of the main theorem}\label{sec:proof}
By Lemma~\ref{lem:pure-type}, it suffices to derive a contradiction in the absolutely continuous case. We first construct finite insertions that will be incompatible with the rigidity in Proposition~\ref{prop:rigidity}. In the planar case, harmonic extensions of smooth boundary functions have finite energy, and we use the Doob--Na\"im formula to compare their energies under finite insertions. In general dimensions, these energies may be infinite, so we use traces of Dirichlet forms on finite sets instead.

\subsection{Finite perturbations}\label{subsec:insertion}
First, we observe that finite perturbations preserve the measure class of harmonic measure: let $\omega,\omega'\in M(X)$ satisfy $\#(\omega\mathbin\triangle\omega')<\infty$. The Voronoi tessellations of $\omega$ and $\omega'$ agree outside
\[
\bigcup_{x\in\omega\setminus\omega'}C_x(\omega)\ \cup\ \bigcup_{x\in\omega'\setminus\omega}C_x(\omega').
\]
This set is compact and meets only finitely many cells of either tessellation. Hence $\#(E_\omega\mathbin\triangle E_{\omega'})<\infty$. Let $S\subset\omega\cup\omega'$ be finite and contain every changed vertex and every endpoint of a changed edge. Then $\omega\setminus S=\omega'\setminus S$, and the neighbors and transition probabilities of every vertex in this common exterior agree. Thus, for every $z\in\omega\setminus S$, we have
\[
\PP_z^\omega(\cdot,\tau_S=\infty)=\PP_z^{\omega'}(\cdot,\tau_S=\infty).
\]

Let $A$ be an event of one-sided paths unchanged by deleting a finite initial segment, then by transience and the Markov property, for every $x\in\omega$ we have
\[
\PP_x^\omega(A)=0
\,\,\Longleftrightarrow\,\,
\PP_z^\omega(A,\tau_S=\infty)=0\text{ for every }z\in\omega\setminus S.
\]
Applying this in both configurations to $A=\{\Xi_+\in C\}$, for Borel $C\subset B$, gives
\[
\nu_{\omega,x}\sim\nu_{\omega',y}
\]
for every $x\in\omega$ and $y\in\omega'$. Combining this observation with Proposition~\ref{prop:rigidity} gives the following lemma.

\begin{lemma}\label{lem:finite-rigidity}
Suppose that $\nu_{\omega,o}\ll m$ for $\mu_o$-almost every $\omega$. Then there exists a $K$-invariant, $\cR$-invariant $\mu_o$-conull Borel set $A\subset M_o(X)$ such that
\[
J_\omega=J_{\omega'}
\]
whenever $\omega,\omega'\in A$ and $\#(\omega\mathbin\triangle\omega')<\infty$.
\end{lemma}
\begin{proof}
Let $A$ be the lift of the complement of the $\cR$-saturation of the exceptional set in Proposition~\ref{prop:rigidity}, and let $\omega,\omega'\in A$ satisfy $\#(\omega\mathbin\triangle\omega')<\infty$. Since finite perturbations preserve the measure class of harmonic measure, Lemma~\ref{lem:current-construction} gives $J_\omega\sim J_{\omega'}$. Proposition~\ref{prop:rigidity} then gives $L(A_\omega\mathbin\triangle A_{\omega'})=0$ and $J_\omega=J_{\omega'}$.
\end{proof}

Before constructing the required insertion, we record the \emph{insertion tolerance} of the Poisson point process, see \cite{HS13} for the general notion.

\begin{lemma}\label{lem:insertion-conull}
Let $A\subset M_o(X)$ be Borel and $K$-invariant, with $\mu_o(A)=1$. Then there exists a $K$-invariant, $\cR$-invariant $\mu_o$-conull Borel set $A_0\subset M_o(X)$ such that, for every $\omega\in A_0$ and every $n\ge1$, we have
\[
\omega\cup\{z_1,\ldots,z_n\}\in A
\]
for $\vol^{\otimes n}$-almost every $(z_1,\ldots,z_n)\in X^n$.
\end{lemma}
\begin{proof}
For $n\ge1$ and a bounded Borel set $U\subset X$, the multivariate Mecke identity \cite[Theorem~4.4]{LP18} gives
\[
\int_{M_o(X)}\int_{U^n}\one_{A^c}(\omega\cup\{z_1,\ldots,z_n\})\,\dd\vol^{\otimes n}(z)\,\dd\mu_o(\omega)=0,
\]
since $\mu_o(A)=1$. By Fubini's theorem, for each $n$ and each bounded $U$ the inner integral vanishes for $\mu_o$-almost every $\omega$. Intersecting these conull sets over $n\ge1$ and $U=B(o,k)$, $k\ge1$, gives a $K$-invariant $\mu_o$-conull Borel set. The lift of the complement of the $\cR$-saturation of the exceptional set in $\Omega_o$ is the required $A_0$.
\end{proof}

For functions $u,v\colon\omega\to\RR$, the energy form on $\mathcal G_\omega$ is given by
\[
\cE_\omega(u,v)=\sum_{\{x,y\}\in E_\omega}
(u(x)-u(y))(v(x)-v(y)),
\]
whenever the sum is absolutely convergent. We write $\cE_\omega(u)=\cE_\omega(u,u)$ for the Dirichlet energy, allowing the value $+\infty$. For a finite insertion $\omega'=\omega\cup Z$ with $E_\omega\mathbin\triangle E_{\omega'}$ finite and a function $u'\colon\omega'\to\RR$, let
\[
\delta\cE_{\omega,\omega'}(u')
=
\sum_{\{a,b\}\in E_{\omega'}\setminus E_\omega}(u'(a)-u'(b))^2
-
\sum_{\{a,b\}\in E_\omega\setminus E_{\omega'}}(u'(a)-u'(b))^2,
\]
which records the change in Dirichlet energy under the insertion.

The following lemma constructs the finite insertions needed for the energy comparison. Since only the conclusion is needed below, the proof may be skipped safely.

\begin{lemma}\label{lem:geometric-insertion}
Let $\omega\subset X$ be locally finite, and let $x\in\omega$ have compact Voronoi cell. For every $r>0$, there exist $n\ge1$ and a nonempty open set $\mathcal U\subset B(x,r)^n$ such that every $(z_1,\ldots,z_n)\in\mathcal U$ consists of distinct points in $X\setminus\omega$ and the configuration $\omega'=\omega\cup\{z_1,\ldots,z_n\}$ satisfies $\#(E_\omega\mathbin\triangle E_{\omega'})<\infty$ and
\[
\delta\cE_{\omega,\omega'}(u')\ge0
\]
for every $u'\colon\omega'\to\RR$. Further, if $\delta\cE_{\omega,\omega'}(u')=0$, then $u'(x)=u'(y)$ for every $y\sim_\omega x$.
\end{lemma}

\begin{proof}
Let $v_1,\ldots,v_k$ be the neighbors of $x$. We will construct distinct inserted points $z_{i,j}$ such that
\[
x\sim_{\omega'}z_{i,j}\sim_{\omega'}v_i
\]
for $1\le i\le k,\,\, 1\le j\le3$, and preserve every old edge not incident to $x$. Given these conditions, we have
\begin{align*}
\delta\cE_{\omega,\omega'}(u')
&\ge\sum_{i=1}^k\sum_{j=1}^3\bigl[(u'(x)-u'(z_{i,j}))^2+(u'(z_{i,j})-u'(v_i))^2\bigr]-\sum_{i=1}^k(u'(x)-u'(v_i))^2\\
&\ge\frac12\sum_{i=1}^k(u'(x)-u'(v_i))^2,
\end{align*}
which would complete the proof.

For each $i$, choose three distinct points $q_{i,1},q_{i,2},q_{i,3}$ in the relative interior of the facet $C_x(\omega)\cap C_{v_i}(\omega)$. Let $z_{i,j}\in[x,q_{i,j}]$ satisfy $d_{\mathbb H}(x,z_{i,j})=\varepsilon$, where the common $\varepsilon\in(0,r)$ will be chosen sufficiently small. The rays from $x$ through the points $q_{i,j}$ are distinct, since each ray meets the boundary of the convex cell $C_x(\omega)$ only once.

Fix $(i,j)$ and write $q=q_{i,j}$, $z=z_{i,j}$. At the midpoint $p$ of $[x,z]$, the triangle inequality and distinctness of the rays give
\[
d_{\mathbb H}(p,x)=d_{\mathbb H}(p,z)=\varepsilon/2<d_{\mathbb H}(p,z_{\ell,m})
\]
for $(\ell,m)\ne(i,j)$.
For small $\varepsilon$, all old sites other than $x$ are farther from $p$, so $x\sim_{\omega'}z$.

Choose $q^+$ on the same ray, slightly beyond $q$, so that
\[
d_{\mathbb H}(q^+,v_i)<d_{\mathbb H}(q^+,x),
\]
and
\[
d_{\mathbb H}(t,y)>\max\{d_{\mathbb H}(t,x),d_{\mathbb H}(t,v_i)\}
\]
for every $t\in[q,q^+]$ and $y\in\omega\setminus\{x,v_i\}$. Again by the triangle inequality and distinctness of the rays,
\[
d_{\mathbb H}(t,z)=d_{\mathbb H}(t,x)-\varepsilon<d_{\mathbb H}(t,z_{\ell,m})
\]
for $t\in[q,q^+]$ and $(\ell,m)\ne(i,j)$.
For $\varepsilon<d_{\mathbb H}(q^+,x)-d_{\mathbb H}(q^+,v_i)$, we have
\[
d_{\mathbb H}(q,z)=d_{\mathbb H}(q,v_i)-\varepsilon,\,\, d_{\mathbb H}(q^+,z)>d_{\mathbb H}(q^+,v_i).
\]
Thus $z$ and $v_i$ are equidistant at some point of $[q,q^+]$ and strictly closer there than every other site, proving $z\sim_{\omega'}v_i$.

Choose $R>0$ with $C_x(\omega)\subset B(x,R)$. Since $d_{\mathbb H}(t,x)>d_{\mathbb H}(t,\omega)$ uniformly on $\partial B(x,R)$, for sufficiently small $\varepsilon$ no new Voronoi cell meets this sphere. By convexity, all new cells lie in $B(x,R)$, so the old edges that may change belong to a finite set independent of $\varepsilon$.

For each such edge $\{a,b\}$ not incident to $x$, choose $p$ in the relative interior of its facet. Then $d_{\mathbb H}(p,a)=d_{\mathbb H}(p,b)<d_{\mathbb H}(p,x)$, so the edge is preserved for sufficiently small $\varepsilon$. Thus we can choose $\varepsilon>0$ so that all the required adjacencies hold and every old edge not incident to $x$ is preserved. These properties are stable under sufficiently small perturbations of the inserted points, giving the desired nonempty open set $\mathcal U\subset B(x,r)^{3k}$.
\end{proof}

\subsection{Hyperbolic plane}\label{sec:plane}
We focus on the case $d=2$ and identify $\mathbb H^2$ with the Poincar\'e disk. We first recall the Dirichlet principle.

\begin{lemma}[Dirichlet principle]\label{lem:dirichlet}
Let $f\colon B\to\RR$ be bounded and Borel. Suppose that $u\colon\omega\to\RR$ is bounded, $\cE_\omega(u)<\infty$, and $u(X_n)\to f(\Xi_+)$ $\PP_x^\omega$-almost surely for every $x\in\omega$. Then
\[
\cE_\omega(H_\omega f)\le\cE_\omega(u).
\]
Further, equality holds if and only if $u=H_\omega f$.
\end{lemma}
\begin{proof}
Let $(F_n)_{n\ge1}$ be an increasing exhaustion of $\omega$ by finite sets, with exit times $\tau_n=\inf\{j\ge0 \colon X_j\notin F_n\}$, which tend to infinity. Let
\[
u_n(x)=\begin{cases}
\EE_x^\omega[u(X_{\tau_n})],&x\in F_n,\\
u(x),&x\notin F_n.
\end{cases}
\]
The first-step decomposition shows that $u_n$ is harmonic on $F_n$, hence $\cE_\omega$-orthogonal to $u-u_n$, which is supported in $F_n$. Therefore,
\[
\cE_\omega(u)=\cE_\omega(u_n)+\cE_\omega(u-u_n).
\]
The convergence $u(X_n)\to f(\Xi_+)$ and bounded convergence give $u_n(x)\to H_\omega f(x)$ for every $x\in\omega$. Then, by Fatou's lemma, we have
\[
\cE_\omega(u)\ge\cE_\omega(H_\omega f)+\cE_\omega(u-H_\omega f).
\]
In the equality case, $u-H_\omega f$ is constant by connectivity, and since $(u-H_\omega f)(X_n)\to0$, the constant is zero.
\end{proof}

\begin{proposition}\label{prop:plane-energy}
For $\mu_o$-almost every $\omega$, we have
\[
\cE_\omega(H_\omega f)<\infty
\]
for every $f\in C^\infty(S^1)$.
\end{proposition}

\begin{proof}
It suffices to prove that, for $\mu_o$-almost every $\omega$,
\[
\sum_{\{x,y\}\in E_\omega}|x-y|^2<\infty,
\]
where $|x-y|$ is Euclidean distance. Indeed, every $f\in C^\infty(S^1)$ admits a $C^\infty$ extension $F$ to the closed disk with finite energy in the usual sense, and the Dirichlet principle gives
\[
\cE_\omega(H_\omega f)\le\cE_\omega(F|_\omega)
\le\|\nabla_{\mathrm{Euc}}F\|_\infty^2\sum_{\{x,y\}\in E_\omega}|x-y|^2<\infty
\]
on this event.

Consider the measure
\[
\frac12\EE_\mu\sum_{x\in\omega}\sum_{y\sim_\omega x}d_{\mathbb{H}}(x,y)\,\delta_x\otimes\bigl(\ell_{\mathbb H}|_{[x,y]}\bigr)
\]
on $\mathbb H^2\times\mathbb H^2$, which is invariant under the diagonal $G$-action. Its first marginal is $\frac{\lambda\ell_2}{2}\vol$, so it is $\sigma$-finite by Lemma~\ref{lem:moments}. By applying the mass-transport principle to $|\nabla_{\mathbb H}F_i|^2$ for the coordinate functions $F_i(z_1,z_2)=z_i$, $i=1,2$, we obtain
\begin{align*}
\EE_\mu\sum_{\{x,y\}\in E_\omega}|x-y|^2
&=\sum_{i=1}^2\EE_\mu\cE_\omega(F_i|_\omega)\\
&\le\frac{\lambda\ell_2}{2}\sum_{i=1}^2\int_{\mathbb H^2}|\nabla_{\mathbb H}F_i|^2\,\dd\vol\\
&=\frac{\lambda\ell_2}{2}\sum_{i=1}^2\int_{\mathbb D}|\nabla_{\mathrm{Euc}}F_i|^2\,\dd z
=\pi\lambda\ell_2<\infty.
\end{align*}
Note that this is where we use $d=2$. Since the event that this sum is finite is $G$-invariant, Campbell's identity shows that it is also $\mu_o$-conull.
\end{proof}

\begin{theorem}[Silverstein]\label{thm:doob-naim}
Let $h$ be a finite-energy harmonic function on $\mathcal G_\omega$. Then the limits $h_\pm=\lim_{n\to\infty}h(X_{\pm n})$ exist for $I_\omega$-almost every trajectory, and
\[
\cE_\omega(h)=\frac12\int_{W_\omega^*}(h_--h_+)^2\,\dd I_\omega(w^*).
\]
\end{theorem}
\begin{proof}
We refer to \cite[Section~3]{Sil74} for the proof.
\end{proof}

Recall $\widetilde J_\omega=(\widetilde\Xi_-,\widetilde\Xi_+)_*I_\omega$ from Subsection~\ref{subsec:naim}. By Proposition~\ref{prop:plane-energy} and Theorem~\ref{thm:doob-naim}, we have
\begin{align*}
\cE_\omega(H_\omega f)
&=\frac12\int_{\widetilde B_\omega^2}
\bigl(f(\pi_\omega(\zeta))-f(\pi_\omega(\zeta'))\bigr)^2
\,\dd\widetilde J_\omega(\zeta,\zeta')\\
&=\frac12\int_{B^2}(f(\xi)-f(\eta))^2
\,\dd(\pi_\omega\times\pi_\omega)_*\widetilde J_\omega(\xi,\eta)\\
&=\frac12\int_{B^{(2)}}(f(\xi)-f(\eta))^2\,\dd J_\omega(\xi,\eta),
\end{align*}
where the last equality follows from Subsection~\ref{subsec:naim}.

\begin{proof}[Proof of Theorem~\ref{thm:main} for $d=2$]
Suppose that $\nu_{\omega,o}\ll m$ for $\mu_o$-almost every $\omega$. We restrict to a $\cR$-invariant conull set on which Lemma~\ref{lem:finite-rigidity}, Proposition~\ref{prop:plane-energy}, and the Doob--Na\"im formula hold. By Lemma~\ref{lem:insertion-conull}, outside a further $\cR$-invariant null set, almost every finite insertion remains in this set. Therefore, for such $\omega$, the nonempty open set of insertion tuples in Lemma~\ref{lem:geometric-insertion} contains an insertion $\omega'=\omega\cup Z$ in the same conull set, for which Lemma~\ref{lem:finite-rigidity} gives $J_{\omega'}=J_\omega$.

For $f\in C^\infty(S^1)$, Proposition~\ref{prop:plane-energy} gives $\cE_\omega(H_\omega f)<\infty$ and $\cE_{\omega'}(H_{\omega'}f)<\infty$. Since only finitely many edges change, we also have $\cE_\omega((H_{\omega'}f)|_\omega)<\infty$. By the martingale convergence theorem and the finite-perturbation argument in Subsection~\ref{subsec:insertion}, we have $H_{\omega'}f(X_n)\to f(\Xi_+)$ $\PP_x^\omega$-almost surely for every $x\in\omega$. By equality of the currents, the Doob--Na\"im formula, and the Dirichlet principle,
\begin{align*}
\cE_\omega(H_\omega f)
&=\cE_{\omega'}(H_{\omega'}f)
=\cE_\omega((H_{\omega'}f)|_\omega)+\delta\cE_{\omega,\omega'}(H_{\omega'}f)\\
&\ge\cE_\omega((H_{\omega'}f)|_\omega)\ge\cE_\omega(H_\omega f).
\end{align*}
Thus equality holds throughout. The Dirichlet principle gives $(H_{\omega'}f)|_\omega=H_\omega f$, and Lemma~\ref{lem:geometric-insertion} gives $H_\omega f(o)=H_\omega f(y)$ for every $y\sim_\omega o$.

Taking a countable dense family in $C(S^1)$ consisting of smooth functions gives $\nu_{\omega,o}=\nu_{\omega,y}$ for every $y\sim_\omega o$, which, by $\cR$-invariance, further implies that $\nu_{\omega,x}$ is independent of $x\in\omega$. By Campbell's identity, the same holds for $\mu$-almost every $\omega$. Writing this common measure as $\eta_\omega$ gives a measurable $G$-equivariant family. Hence $\int\eta_\omega\,\dd\mu(\omega)$ is a $G$-invariant probability measure on $B$, which is a contradiction.
\end{proof}

\subsection{General case}\label{sec:higher}
In higher dimensions, harmonic extensions of boundary functions need not have finite Dirichlet energy, so the preceding argument using the Doob--Na\"im formula cannot be applied directly. To overcome this difficulty, we use the trace forms on finite sets and their relation to Na\"im currents. See also \cite{CFY06} for background and general treatment.

Let $F\subset\omega$ be finite and nonempty. The walk observed at successive visits to $F$ is defined by the (substochastic) transition kernel
\[
p_{\omega,F}(s,t)=\PP_s^\omega(\tau_F^+<\infty,\ X_{\tau_F^+}=t), 
\]
for $s,t\in F$, where $\tau_F^+$ is defined in Subsection~\ref{subsec:naim}. The associated form is defined by
\[
\cE_{\omega,F}(a,b)=\frac12\sum_{s,t\in F}\deg_\omega(s)p_{\omega,F}(s,t)(a(s)-a(t))(b(s)-b(t))
+\sum_{s\in F}\rho_{\omega,F,s}(B)a(s)b(s),
\]
for $a,b\colon F\to\RR$, where $\rho_{\omega,F,s}$ is defined before Lemma~\ref{lem:current-construction}. We write $\cE_{\omega,F}(a)=\cE_{\omega,F}(a,a)$ and also use the same notation for functions on $\omega$ by restriction to $F$.

This kernel is reversible, and the first-step decomposition gives
\[
\deg_\omega(s)p_{\omega,F}(s,t)=\one_{\{s\sim_\omega t\}}+
\sum_{\substack{y\sim_\omega s \\ y\notin F}}\PP_y^\omega(\tau_F<\infty,\ X_{\tau_F}=t)
\]
and
\[
\rho_{\omega,F,s}(B)=\deg_\omega(s)\left(1-\sum_{t\in F}p_{\omega,F}(s,t)\right),
\]
and hence
\[
\cE_{\omega,F}(a,b)=\sum_{s\in F}\deg_\omega(s)b(s)
\left(a(s)-\sum_{t\in F}p_{\omega,F}(s,t)a(t)\right).
\]
The form $\cE_{\omega,F}$ is positive definite: the trace kernel is irreducible, and transience gives a positive killing probability at some vertex of $F$.

Let $f\colon B\to\RR$ be bounded and Borel. By the strong Markov property, we have
\[
H_\omega f(s)=\sum_{t\in F}p_{\omega,F}(s,t)H_\omega f(t)
+\frac{\rho_{\omega,F,s}(f)}{\deg_\omega(s)}
\]
for $s\in F$. By combining this with Lemma~\ref{lem:current-construction}, we obtain
\[
\cE_{\omega,F}(H_\omega f,v)=\sum_{s\in F}\rho_{\omega,F,s}(f)v(s),\,\,
\cE_{\omega,F}(H_\omega f)=\sum_{s\in F}\rho_{\omega,F,s}(f)\nu_{\omega,s}(f)=J_{\omega,F}(f\otimes f),
\]
for every $v\colon F\to\RR$. These identities involve only finite quantities for bounded $f$, even if the global energy $\cE_\omega(H_\omega f)$ is infinite.

For $a\colon F\to\RR$ and $y\in\omega$, let
\[
h_a(y)=\EE_y^\omega\!\left[a(X_{\tau_F})\one_{\{\tau_F<\infty\}}\right].
\]
Then $h_a$ has finite energy and
\[
\cE_\omega(h_a)=\cE_{\omega,F}(a).
\]
Indeed, for a finite exhaustion $(F_n)_{n\ge1}$ of $\omega$ containing $F$, let $\sigma_n$ be the exit time from $F_n$ and put
$h_n(y)=\EE_y^\omega[a(X_{\tau_F})\one_{\{\tau_F<\sigma_n\}}]$.
The functions $h_n$ are finitely supported, harmonic on $F_n\setminus F$, and converge pointwise to $h_a$. For $m\ge n$, harmonicity gives
$\cE_\omega(h_n)=\cE_\omega(h_m)+\cE_\omega(h_n-h_m)$.
Since $\cE_\omega(h_n)$ is decreasing, Fatou's lemma gives
\[
\cE_\omega(h_n-h_a)
\le\liminf_{m\to\infty}\cE_\omega(h_n-h_m)
=\cE_\omega(h_n)-\lim_{m\to\infty}\cE_\omega(h_m)\to0,
\]
and hence $h_n\to h_a$ in energy. Since $h_n=a$ on $F$ and $h_n$ is harmonic on $F_n\setminus F$, it holds that
\[
\cE_\omega(h_n)=\sum_{s\in F}a(s)\sum_{y\sim_\omega s}(a(s)-h_n(y)),
\]
and further we can pass to the limit and obtain
\begin{align*}
\cE_\omega(h_a)
&=\lim_{n\to\infty}\cE_\omega(h_n)
=\sum_{s\in F}a(s)\sum_{y\sim_\omega s}(a(s)-h_a(y))\\
&=\sum_{s\in F}\deg_\omega(s)a(s)
\left(a(s)-\sum_{t\in F}p_{\omega,F}(s,t)a(t)\right)=\cE_{\omega,F}(a).
\end{align*}

Based on these preparations, we can use the Na\"im current to compare energies under finite insertions.

\begin{lemma}\label{lem:finite-comparison}
Let $\omega,\omega'\in M(X)$ satisfy $\omega\subseteq\omega'$, $\#(\omega'\setminus\omega)<\infty$, and $\#(E_\omega\mathbin\triangle E_{\omega'})<\infty$. Suppose that $J_\omega=J_{\omega'}$ and $J_\omega$ is Radon on $B^{(2)}$. Then, for every bounded Borel function $f\colon B\to\RR$, the function $u=(H_{\omega'}f)|_\omega-H_\omega f$ has finite energy and satisfies
\[
\cE_\omega(u)+\delta\cE_{\omega,\omega'}(H_{\omega'}f)=0.
\]
In particular, if $\delta\cE_{\omega,\omega'}(v)\ge0$ for every $v\colon\omega'\to\RR$, then
\[
(H_{\omega'}f)|_\omega=H_\omega f,\,\, \delta\cE_{\omega,\omega'}(H_{\omega'}f)=0
\]
for every bounded Borel function $f\colon B\to\RR$.
\end{lemma}

\begin{proof}
Choose a finite set $S\subset\omega$ containing every vertex of $\omega$ incident to a changed edge, and put $S'=S\cup(\omega'\setminus\omega)$, so that $\omega\setminus S=\omega'\setminus S'$. Trajectories staying in this common exterior contribute equally to $J_\omega$ and $J_{\omega'}$. Since $J_\omega=J_{\omega'}$ is Radon on $B^{(2)}$, we can subtract this common contribution on compact subsets of $B^{(2)}$ and obtain
\[
J_{\omega,S}|_{B^{(2)}}=J_{\omega',S'}|_{B^{(2)}}.
\]
Moreover, for $s\in S$, the first-step decomposition gives
\[
\rho_{\omega,S,s}
=\sum_{\substack{y\sim_\omega s\\y\notin S}}
\PP_y^\omega(\Xi_+\in\cdot,\tau_S=\infty)
=\rho_{\omega',S',s},
\]
while $\rho_{\omega',S',z}=0$ for $z\in S'\setminus S$. By Lemma~\ref{lem:current-construction}, the finite currents therefore have the same first marginal. Their difference is supported on the diagonal and has zero first marginal, so
\[
J_{\omega,S}=J_{\omega',S'}
\]
on $B^2$.
Let $h=H_\omega f$, $h'=H_{\omega'}f$, and $u=h'|_\omega-h$. Then, by the identities for the trace form above, we have $\cE_{\omega',S'}(h')=\cE_{\omega,S}(h)$.
Moreover, the identities above give
\[
\sum_{s\in S}\rho_{\omega,S,s}(f)h'(s)
=J_{\omega',S'}(f\otimes f)=J_{\omega,S}(f\otimes f)
=\sum_{s\in S}\rho_{\omega,S,s}(f)h(s).
\]
Hence we have
\[
\cE_{\omega,S}(h,u)=\sum_{s\in S}\rho_{\omega,S,s}(f)u(s)=0.
\]
For $a\colon S\to\RR$, the first-step decomposition gives
\begin{align*}
\cE_{\omega,S}(a)
&=\sum_{\substack{\{s,t\}\in E_\omega\\s,t\in S}}(a(s)-a(t))^2+\frac12\sum_{s,t\in S}\sum_{\substack{y\sim_\omega s\\y\notin S}}
\PP_y^\omega(\tau_S<\infty,\ X_{\tau_S}=t)(a(s)-a(t))^2\\
&\,\,+\sum_{s\in S}\rho_{\omega,S,s}(B)a(s)^2.
\end{align*}
By the choice of $S$, the second and third terms are the same in the corresponding expression for $\cE_{\omega',S'}$, and hence we have
\[
\cE_{\omega',S'}(v)-\cE_{\omega,S}(v|_S)=\delta\cE_{\omega,\omega'}(v)
\]
for $v\colon S'\to\RR$.
Applying this with $v=h'|_{S'}$ and using the preceding identities, we obtain
\begin{align*}
0&=\cE_{\omega',S'}(h')-\cE_{\omega,S}(h)
=\cE_{\omega,S}(h+u)-\cE_{\omega,S}(h)+\delta\cE_{\omega,\omega'}(h')\\
&=\cE_{\omega,S}(u)+2\cE_{\omega,S}(h,u)+\delta\cE_{\omega,\omega'}(h')
=\cE_{\omega,S}(u)+\delta\cE_{\omega,\omega'}(h').
\end{align*}
For $y\in\omega\setminus S$, the strong Markov property gives
\begin{align*}
h(y)
&=\EE_y^\omega\!\left[h(X_{\tau_S})\one_{\{\tau_S<\infty\}}\right]
+\EE_y^\omega\!\left[f(\Xi_+)\one_{\{\tau_S=\infty\}}\right],\\
h'(y)
&=\EE_y^\omega\!\left[h'(X_{\tau_S})\one_{\{\tau_S<\infty\}}\right]
+\EE_y^\omega\!\left[f(\Xi_+)\one_{\{\tau_S=\infty\}}\right].
\end{align*}
From these equalities we have
\[
u(y)=\EE_y^\omega\!\left[u(X_{\tau_S})\one_{\{\tau_S<\infty\}}\right]
\]
for every $y\in\omega$.
Applying the identity for $h_a$ with $F=S$ and $a=u|_S$ gives $\cE_\omega(u)=\cE_{\omega,S}(u)<\infty$, proving the asserted identity. If $\delta\cE_{\omega,\omega'}\ge0$, positive definiteness implies $u|_S=0$ and $\delta\cE_{\omega,\omega'}(h')=0$. The equality for $u(y)$ then gives $u=0$.
\end{proof}

\begin{proof}[Proof of Theorem~\ref{thm:main}]
By Lemma~\ref{lem:pure-type}, it suffices to derive a contradiction from the assumption that $\nu_{\omega,o}\ll m$ for $\mu_o$-almost every $\omega$. Fix a $K$-invariant, $\cR$-invariant $\mu_o$-conull set $A$ on which Lemma~\ref{lem:finite-rigidity} holds and the currents are Radon, as provided by Proposition~\ref{prop:local-finiteness}. Further, let $A_0\subset A$ be the $\cR$-invariant conull set given by Lemma~\ref{lem:insertion-conull}. For every $\omega\in A_0$, Lemma~\ref{lem:insertion-conull} and Lemma~\ref{lem:geometric-insertion} give a finite insertion $\omega' = \omega \cup Z \in A$ satisfying the conclusion of Lemma~\ref{lem:geometric-insertion}. Lemma~\ref{lem:finite-rigidity} gives $J_\omega=J_{\omega'}$, and Lemma~\ref{lem:finite-comparison} therefore yields
\[
(H_{\omega'}f)|_\omega=H_\omega f,\,\, \delta\cE_{\omega,\omega'}(H_{\omega'}f)=0
\]
for every bounded Borel function $f\colon B\to\RR$. The equality condition in Lemma~\ref{lem:geometric-insertion} gives $H_\omega f(o)=H_\omega f(y)$ for every $y\sim_\omega o$. Hence $\nu_{\omega,o}=\nu_{\omega,y}$ for every neighbor $y$ of $o$. By $\cR$-invariance and connectivity, $\nu_{\omega,x}$ is independent of $x\in\omega$. By Campbell's identity, averaging the resulting $G$-equivariant family of harmonic measures with respect to $\mu$ gives a $G$-invariant probability measure on $B$, which is impossible as in the planar case.
\end{proof}

\bibliographystyle{amsalpha}
\bibliography{main}
\end{document}